\documentclass[a4paper,11pt]{amsart}
\usepackage[utf8]{inputenc}
\usepackage[english]{babel}
\usepackage{amsaddr}
\usepackage{calrsfs}
\usepackage{amssymb}
\usepackage{amsmath}
\usepackage{mathrsfs}
\usepackage{bm}
\usepackage[toc,page,header]{appendix}
\usepackage{graphicx,color}
\usepackage[dvipsnames]{xcolor}
\usepackage[colorlinks=true,
]{hyperref}

\usepackage{tikz}
\usetikzlibrary{decorations.markings}
\usepackage{booktabs}
\usepackage{multirow}
\usepackage{float}
\usepackage{array}

\makeatletter
\@namedef{subjclassname@2020}{%
	\textup{2020} Mathematics Subject Classification}
\makeatother

\theoremstyle{definition}
\newtheorem{dfn}{Definition}
\theoremstyle{plain}
\newtheorem{thm}{Theorem}
\newtheorem{lmm}{Lemma}
\theoremstyle{remark}
\newtheorem{rem}{Remark}

\newcommand{\Z}{\mathbb{Z}}
\newcommand{\E}{\mathbb{E}}

\newcommand{\R}{\mathbb{R}}

\newcommand{\e}{\mathrm{e}}

\newcommand{\lb}{\lbrace}
\newcommand{\rb}{\rbrace}

\renewcommand{\d}{\mathrm{d}}
\newcommand{\p}{\partial}

\newcommand{\bH}{\mathbf{H}}
\newcommand{\bsigma}{\bm{\sigma}}
\newcommand{\ba}{\mathbf{a}}
\newcommand{\bb}{\mathbf{b}}
\newcommand{\brho}{\bm{\rho}}
\newcommand{\bg}{\mathbf{g}}

\begin{document}
\title[Parameter Estimation for mmfOU by GMM]{Parameter Estimation for multi-mixed Fractional Ornstein--Uhlenbeck Processes by Generalized Method of Moments}

\date{\today}

\author[Almani]{Hamidreza Maleki Almani}
\address{School of Technology and Innovations, University of Vaasa, P.O. Box 700, FIN-65101 Vaasa, FINLAND}
\email{hmaleki@uwasa.fi}

\author[Sottinen]{Tommi Sottinen}
\address{School of Technology and Innovations, University of Vaasa, P.O. Box 700, FIN-65101 Vaasa, FINLAND}
\email{tommi.sottinen@uwasa.fi}


\begin{abstract}
We develope the generalized method of moments (GMM) estimation for the parameters of the finitely mixed multi-mixed fractional Ornstein--Uhlenbeck (mmfOU) processes, and analyze the consistency and assymptotic normality of this estimator. We also include some simulations and provide numerical observations considering different statistical errors.
\end{abstract}



\maketitle


\section{Introduction and Preliminaries}\label{sect:introduction}
%
Langevin \cite{Langevin-1908} postulated the model
\begin{equation}\label{eq: Langevin}
	\frac{\d U_t}{\d t} = -\frac{f}{m}U_t + \frac{F_t}{m},
\end{equation}
to the velocity $U$ of a free particle with mass $m>0$, moving in a liquid by the alternating force $F$ of surrounding molecules, and a constant friction coefficient $f$. By imposing probabilistic hypotheses to the force $F$, Ornstein and Uhlenbeck \cite{Ornstein-Uhlenbeck-1930} derived that the solution of \eqref{eq: Langevin} has a Gaussian distribution with exponential mean. Later on, considering the initial variable $U_0$ to be central Gaussian and independent from $F$, Doob \cite{Doob-1942} showed that the solution to \eqref{eq: Langevin} is stationary and for some constant $c>0$, the time-changed scaled process  $t^{1/2}\,U(c\log t)$ is in fact the well-known Einstein's \cite{Einstein-BrownianMotion-1905} Brownian motion for $t\geq 0$. As the Brownian motion in nowhere differentiable, Doob understood \eqref{eq: Langevin} as
\begin{equation}\label{eq: Doob_Langevin}
	\d U_t = -\lambda U_t\d t + \d X_t,
\end{equation}
where $U_0=\xi$ is the initial variable and $X$ is a L\'evy process. He derived the solution
\begin{equation}\label{eq: OU_Levy}
	U_t = \e^{-\lambda t}\xi +\int_{0}^t \e^{-\lambda(t-s)}\, \d X_s,
\end{equation}
which is called the Ornstein--Uhlenbeck (OU) process. This solution was later extent for all cases that $X$ is a semimartingale (see e.g. \cite{Protter-1992}). Lamperti \cite{Lamperti-1962} proved the processes $V_t, t\geq 0$ is $H$-selfsimilar if and only if $U_t = e^{-\lambda t}V(c\exp(\lambda t/H))$ is stationary for all $\lambda , c>0$. 
Cheridito and et al.  \cite{Cheridito-Kawaguchi-Maejima-2003} applied Lamperti transform and verified the solution \eqref{eq: OU_Levy} is valid for $X=B^H$, where $B^H$ is the fractional Brownian mtoion (fBm) introduced by Mandelbrot and Van Ness \cite{Mandelbrot-Van-Ness-1968}.

In \cite{almani_sottinen_2023} the Langevin equation \eqref{eq: Doob_Langevin} with $X = \sum_{i=1}^\infty \sigma_i B^{H_i}$, the multi-mixed fractional Brownian motion (mmfBm), is the dynamic of particle's velocity when the liquid is not purely homogeneous, i.e. the local surrounding molecules can have imposing forces of different roughnesses. That is
\begin{equation}\label{eq: Almani_Langevin}
	\d U_t = -\lambda U_t\d t + \d\left(\sum_{i=1}^\infty \sigma_i B^{H_i}_t\right).
\end{equation}
They confirmed the solution \eqref{eq: OU_Levy} for this model is possible if for all $i\ge 1$, $H_i\in [H_{\inf},H_{\sup}]\subset (0,1)$ and $\sum_{i=1}^\infty\sigma_i^2<\infty$. The solution is
\begin{equation}\label{eq: mmfOU}
	U_t = \e^{-\lambda t}\xi +\int_{0}^t \e^{-\lambda(t-s)}\, \d\left(\sum_{i=1}^\infty \sigma_i B^{H_i}_t\right).
\end{equation}
When $U_0=\xi=\sum_i \sigma_i\xi_i$ is the initial condition, this solution can also be written as
\begin{equation}\label{eq: plausibility}
	U_t = \sum_{i=1}^\infty\sigma_i U^{\lambda,H_i}_t,
\end{equation}
where
\begin{equation}\label{eq: Cheridito_Langevin}
	\d U^{\lambda,H_i}_t = -\lambda U^{\lambda,H_i}_t \d t + \d B^{H_i}_t,
\end{equation}
and $U^{\lambda,H_i}_0 =\xi_i$. To approximate $U$, one can naturally consider
\begin{equation}\label{eq: mmfOU_n}
	U^n_t = \sum_{i=1}^n\sigma_i U^{\lambda,H_i}_t,
\end{equation}
that is convergent to $U$ in $L^2(\Omega\times[0,T])$, for any $T\geq 0$ (see \cite{almani_sottinen_2023}). To apply the process \eqref{eq: mmfOU_n} for modelling the real data, we need $2n+1$ real-valued parameters
$$\lambda;\, H_1,\ldots,H_n;\,\sigma_1,\ldots,\sigma_n$$
to be statistically estimated. 
To do this, we use the generalized method of moments (GMM) introduced in Barbiza and Viens \cite{Barboza-Viens-2017}.\\

The outline of this paper is as follows. In Section~\ref{sect:estimation} we develop the GMM estimator for the process \eqref{eq: mmfOU_n} analytically. Next, in Section~\ref{sect:consistency}, we investigate the consistency of the introduced estimator. In Section~\ref{sect:normality} we prove the asymptotic normality of our estimator. Finally, in Section~\ref{sect:simulation}, we provide some simulations.
\section{Generalized Methdod of Moments Estimation}\label{sect:estimation}

%
The parameters we aim to estimate are
\begin{equation*}
\lambda>0,\qquad 
H_k \in [H_{\inf}, H_{\sup}] \subset (0,1), \qquad
\sigma_k>0; \qquad
k=1,2,\ldots n,
\end{equation*}
i.e. the following vector will be estimated
\begin{align*}
&\theta=(\lambda ;\bH ;\bsigma)'\in\R^{2n+1}\\
&\mbox{where}\\
&\bH=(H_1,\ldots,H_n),\\
&\bsigma=(\sigma_1,\ldots,\sigma_n).
\end{align*}
To develope GMM estimator, the autocovariance function of $U^n$ is required. It is given in \cite{almani_sottinen_2023} by
\begin{eqnarray}\label{eq: rho_analitical}
&&\rho_\theta(t) = \rho_{\lambda,n}(t) = \E[U^n_{s}U^n_{s+t}]\\
&=&
\sum_{k=1}^n\sigma_k^2\frac{\Gamma(1+2H_k)}{4}\frac{\e^{-\lambda t}}{\lambda^{2H_k}}\bigg\lb 1+{\gamma}_{2H_k-1}(\lambda t)+\e^{2\lambda t}{\Gamma}_{2H_k-1}(\lambda t)\bigg\rb\notag\\
&=&  
\sum_{k=1}^n\sigma_k^2\frac{\Gamma(2H_k+1)}{2\lambda^{2H_k}}\,\bigg\{\cosh(\lambda t) 
-
\frac{1}{\Gamma(2H_k)}\int_0^{\lambda t}s^{2H_k-1}\cosh(\lambda t-s)\, ds\bigg\},\notag 
\end{eqnarray}
where for $\alpha\in (-1,0)\cup (0,1)$
\begin{eqnarray*}
	{\gamma}_\alpha(x) &=& \frac{1}{\Gamma(\alpha)}\int_0^x s^{\alpha-1}\e^{s}\, \d s,\label{NLG} \\ 
	{\Gamma}_\alpha(x) &=& \frac{1}{\Gamma(\alpha)}\int_x^\infty s^{\alpha-1}\e^{-s}\, \d s,\label{NUG} 
\end{eqnarray*}	
and ${\gamma}_0(x)=1$, ${\Gamma}_0(x)=0$. Here ${\gamma}_\alpha$, ${\Gamma}_\alpha$ are the normalized incomplete gamma functions (see \cite{abramowitz1988handbook}). The formula \eqref{eq: rho_analitical} is continuous in $t$, and so $U^n$ is ergodic (see \cite{grenander1950stochastic}, page 257). By Lemma 5.2 in \cite{hu2010parameter} for sufficiently small $t\searrow 0$ we have
\begin{equation}\label{eq: rho_practical}
   \rho_\theta(t)\approx\sum_{k=1}^n\sigma_k^2\left(\frac{H\Gamma(2H)}{\lambda^{2H}} - \frac{t^{2H}}{2}\right).
\end{equation}
 Further, for our estimation, we need some filters associated to the parameter $\theta$ and a positive integers $L\ge 2n+1$. We define them as following.
\begin{dfn}
A filter of length $L+1$ and order $l$ is a sequence of real numbers $\ba=(a_k)_{k=0}^L=(a_0,\ldots,a_L)$, that for $l>0$
\begin{align*}
&\sum_{q=0}^L a_q q^r=0\,;\quad 0\le r\le l-1,\ r\in\Z,\\
&\sum_{q=0}^L a_q q^l\ne 0.
\end{align*}
When $l=0$, we put $a_0=1$ and $a_q = 0$ for $0 < q \le L$.
\end{dfn}

The GMM estimation requires a family of filters by the length $L$ and orders $l=0,1,\ldots,L$. For each filter $\ba$ with order $l$ define
\begin{equation*}
\bb=(b_0,\ldots,b_L)',
\end{equation*}
where
\begin{align*}
&b_0:=\sum_{q=0}^L a_q^2,\\
&b_k:=2\sum_{q=0}^{L-k} a_{q+k}a_q\,;\quad k=1,\ldots,L.
\end{align*}

For different orders $l_i\ne l_j$, the corresponding vectors $\bb_i$ and $\bb_j$ are linearly independent.  So, we can take $L$ filters $\ba_1,\ldots,\ba_L$ with respective orders $l_1,\ldots,l_L$ such that the $L \times (L+1)$ matrix $B=[\bb_1\ \bb_2\ \cdots\ \bb_L]'$ satisfies $2n+1\leq \mbox{rank}(B)\leq L$. 
Then, for a fixed step size $\alpha>0$, and $\theta\in\Theta$ (a closed subset of $\R^{2n+1}$), we can define the function
\begin{equation*}
V(\theta):=\sum_{k=0}^L b_k\rho_\theta (\alpha k)=\sum_{k=0}^L b_k\rho_{\lambda,n} (\alpha k),
\end{equation*}
and the filtered process of the step size $\alpha>0$ for $t\ge 0$ 
\begin{equation*}
\varphi(t):=\sum_{q=0}^L a_qU^n_{t-q\alpha}.
\end{equation*}

Next, let $\theta_0$ be the exact value of the parameter. Due to the stationarity of $U^n$, the random variable $\varphi(t)$ is centered Gaussian with variance $\E[\varphi(t)^2] = V (\theta_0)$.
Following the notation of \cite{newey1994large}, for any $\theta\in\Theta$ we have the vector of differences
\begin{equation*}
\bg(t,\theta):=\left(g_1(t,\theta),\ldots,g_L(t,\theta)\right)',
\end{equation*}
where
\begin{equation*}
g_\ell(t,\theta)=\varphi_\ell(t)^2-V_\ell(\theta),\quad 1\leq \ell\leq L,
\end{equation*}
and the subscript $\ell$ means that we use the filter $\ba_\ell$ in the
filtered process and variance. Moreover, $\bg(t,\theta)$ satisfies a {\it population moment condition} (see \cite{hall2005generalized}
and \cite{newey1994large})
\begin{equation*}
\E[\bg(t,\theta_0)]=0,\quad t\ge 0.
\end{equation*}
Also, the vector $\bg$ is in the  second Wiener chaos (see \cite{nourdin2012normal}).\\

Further, assume that we have observed the process $U^n$ at discrete equidistant times $0=t_0<t_1<\cdots <t_{N-1}<t_{N}=T$ with fixed time step $\alpha = t_i-t_{i-1}$. Let $A$ be an $L\times L$ symmetric positive-definite matrix (a suitably chosen $A$ ensures that the GMM estimate is efficient, see \cite{Barboza-Viens-2017} Section
2.3). Now, for $\theta\in\Theta$ and time $t \ge 0$ denote\\
\begin{enumerate}
\item[] $\bg_0(\theta):=\E[\bg(t,\theta)]$ (vector of expected differences),
\item[] $\hat{\bg}_N(\theta):=\frac{1}{N-L+1}\sum_{i=L}^N\bg(t_i,\theta)$ (vector of sampled differences),
\item[] $Q_0(\theta):=\bg_0(\theta)'A\bg_0(\theta)$ (squared distance of expected differences),
\item[] $\hat{Q}_N(\theta):=\hat{\bg}_N(\theta)'A\hat{\bg}_N(\theta)$ (sampled version of the above distance).\\
\end{enumerate}
Note that $\bg_0(\theta)$ does not depend on time $t$, due to the stationarity 
of $U^n$. Finally, we define the GMM estimator
of $\theta_0=(\lambda_0;\bH_0;\bsigma_0^2)$ as
\begin{equation}\label{estimator}
\hat{\theta}_N=(\hat{\lambda}_N;\hat{\bH}_N;\hat{\bsigma}_N)'=\arg\min_{\theta\in\Theta}\hat{Q}_N(\theta).
\end{equation}
In fact, if the function $Q_0(\theta)$ attains a unique zero at $\theta_0$, then we can find a value $\hat{\theta}_N$ such that the approximation of $Q_0$, that is $\hat{Q}_N$, should be the smallest possible (see \cite{hall2005generalized,hansen1982large}).\\

Moreover, all the remaining results are still valid if we substitute
the fixed matrix $A$ with a sequence of random matrices (perhaps depending
on the data) with a deterministically bounded eigenstructure. In particular, we
can choose such sequence in order to attain convergence in probability to the
efficient alternative of $A$ (see \cite{Barboza-Viens-2017}, Section 2.3).\\

In the following we will show the consistency and asymptotic normality of the GMM estimator \eqref{estimator}.

\section{Consistency}\label{sect:consistency}

As mentioned in \cite{Barboza-Viens-2017}, for the consistency, it is sufficient to find conditions under which the mapping $\theta\mapsto\brho_{\theta,2n+1}(\alpha)$ is smooth and invertible at least in a closed subset $\theta\in\Theta\subseteq\mathbb{R}^{2n+1}$, where
$$
\brho_{\theta,2n+1}(\alpha)=(\rho_\theta(0\cdot\alpha),\ldots,\rho_\theta((2n)\cdot\alpha))',
$$
and $\alpha>0$ is the step size.
This is equivalent for $\nabla_\theta\brho_{\theta,2n+1}(\alpha)$ to be non-singular in a closed subset $\theta\in\Theta\subseteq\mathbb{R}^{2n+1}$.
To show this for the mmfOU processes, with 
$\theta=(\lambda ;\bH ;\bsigma)'\in\R^{2n+1}$,
we have the next lemma and theorem.
\begin{lmm}\label{lmm: clarified}
In the closed rectangle
\begin{equation}
	\Upsilon_3 = \Bigg\lb \theta=(\lambda,H,\sigma)'\,\Bigg\vert\,
	\begin{matrix}
		\lambda,\sigma>0,\,
		H\in (0,1),\\
		\lambda<\exp\big(\Psi(2H+1)\big)
	\end{matrix}
	\Bigg\rb\subset\Theta\subseteq\R^3,
\end{equation}
where $\Psi(\cdot)$ is the so called digamma function
\[\Psi(z)=\frac{\d}{\d z}\log\Gamma(z)=\frac{\Gamma'(z)}{\Gamma(z)},\]
for $\rho_\theta(t) = \sigma^2\rho_{\lambda,H}(t)$, the matrix
\begin{equation*}
	\nabla_\theta\brho_{\theta,3}(\alpha) = 
	\left[
	\begin{array}{lll}
		\frac{\p\rho_\theta}{\p\lambda}(0) & \frac{\p\rho_\theta}{\p H}(0) & \frac{\p\rho_\theta}{\p\sigma}(0)\\
		& & \\
		\frac{\p\rho_\theta}{\p\lambda}(\alpha) & \frac{\p\rho_\theta}{\p H}(\alpha) & \frac{\p\rho_\theta}{\p\sigma}(\alpha)\\
		& & \\
		\frac{\p\rho_\theta}{\p\lambda}(2\alpha) & \frac{\p\rho_\theta}{\p H}(2\alpha) & \frac{\p\rho_\theta}{\p\sigma}(2\alpha)\\
	\end{array}
	\right]
\end{equation*}
is a P-matrix for sufficiently small $\alpha$.
\end{lmm}
\begin{proof}
	Proof is similar to Lemma B.2 in \cite{Barboza-Viens-2017}.
\end{proof}
\begin{thm}\label{thm: cncty}
	In the closed rectangle
\begin{equation}
	\Upsilon_{2n+1} = \Bigg\lb \theta=(\lambda ;\bH ;\bsigma)'\,\Bigg\vert\,
	\begin{matrix}
	\lambda,\sigma_k>0,\,
	H_k\in (0,1),\\
	\lambda<\exp\big(\Psi(2H_{\inf}+1)\big)
\end{matrix}
	\Bigg\rb\subset\Theta\subseteq\R^{2n+1},
\end{equation}
	the mapping
	\begin{equation}\label{eq: Mapping}
\brho_{\theta,2n+1}(\alpha)=(\rho_\theta(0\cdot\alpha),\ldots,\rho_\theta((2n)\cdot\alpha))',
	\end{equation}
	is injective for sufficiently small $\alpha$.
\end{thm}
\begin{proof}
	Let $\theta\in\Upsilon_{2n+1}$, and $x=(x_1,\ldots,x_{2n+1})'>0$ is a positive vector in $\R^{2n+1}$. Then, the $k$th component of the vector $\nabla_\theta\brho_{\theta,2n+1}(\alpha)\,x$ is
	\begin{align}\label{eq: VecProd}
		&\Big[\nabla_\theta\brho_{\theta,2n+1}(\alpha)\,x\Big]_k = \sum_{j=1}^{2n+1}\frac{\p\rho_{\lambda,n}}{\p\theta_j}((k-1)\alpha)x_j\\ \nonumber
		&= \frac{\p\rho_{\lambda,n}}{\p\lambda}((k-1)\alpha)\,x_1
		+ \sum_{j=1}^n\frac{\p\rho_{\lambda,n}}{\p H_j}((k-1)\alpha)\,x_{j+1} \\ \nonumber
		&+ \sum_{j=1}^n\frac{\p\rho_{\lambda,n}}{\p \sigma_j}((k-1)\alpha)\,x_{n+j+1}\notag\\
		\nonumber
		&= \sum_{j=1}^n\sigma_j^2\frac{\p\rho_{\lambda,H_j}}{\p\lambda}((k-1)\alpha)\,x_1
		+ \sum_{j=1}^n\sigma_j^2\frac{\p\rho_{\lambda,H_j}}{\p H_j}((k-1)\alpha)\,x_{j+1} \\ \nonumber
		&+ \sum_{j=1}^n2\sigma_j\frac{\p\rho_{\lambda,H_j}}{\p \sigma_j}((k-1)\alpha)\,x_{n+j+1}\notag\\ \nonumber
		&= \sum_{j=1}^n\bigg[\sigma_j^2\frac{\p\rho_{\lambda,H_j}}{\p\lambda}((k-1)\alpha)\,x_1
		+ \sigma_j^2\frac{\p\rho_{\lambda,H_j}}{\p H_j}((k-1)\alpha)\,x_{j+1} \\  \nonumber
		&+ 2\sigma_j\frac{\p\rho_{\lambda,H_j}}{\p \sigma_j}((k-1)\alpha)\,x_{n+j+1}\bigg].\notag
	\end{align}
Now, for all $j=1,\ldots,2n+1$, let $\tilde\rho_j(t)=\sigma_j^2\rho_{\lambda,H_j}(t)$ and 
\begin{equation*}
	\tilde\alpha = \bigg\lb
	\begin{array}{lll}
		(k-1)\alpha &;& 1\le k\le n,\\
		(k-1)\alpha/2 &;& n+1\le k\le 2n+1.
	\end{array}
\end{equation*}
Then, since $\exp(\Psi(\cdot))$ is an strictly increasing function,
 $$\lambda<\exp\big(\Psi(2H_{\inf}+1)\big)=\min_{1\le i\le 2n+1}\exp\big(\Psi(2H_i+1)\big)\le\exp\big(\Psi(2H_j+1)\big),$$ 
 and so by Lemma \ref{lmm: clarified} the matrix
\begin{equation}\label{eq: SubGrad}
	\nabla_{_{\lambda,H_j,\sigma_j}}\tilde\rho_j(\tilde\alpha) = 
	\left[
	\begin{array}{*3{>{\displaystyle}c}p{5cm}}
		\frac{\p\tilde\rho_j}{\p\lambda}(0) & \frac{\p\tilde\rho_j}{\p H_j}(0) & \frac{\p\tilde\rho_j}{\p\sigma_j}(0)\\
		 & & \\
		\frac{\p\tilde\rho_j}{\p\lambda}(\tilde\alpha) & \frac{\p\tilde\rho_j}{\p H_j}(\tilde\alpha) & \frac{\p\tilde\rho_j}{\p\sigma_j}(\tilde\alpha)\\
		 & & \\
		\frac{\p\tilde\rho_j}{\p\lambda}(2\tilde\alpha) & \frac{\p\tilde\rho_j}{\p H_j}(2\tilde\alpha) & \frac{\p\tilde\rho_j}{\p\sigma_j}(2\tilde\alpha)\\
	\end{array}
	\right]
\end{equation}
is a P-matrix for all $\tilde\alpha\le\tilde\alpha_{kj}$ sufficiently small associated to $k$ and $\tilde\rho_j$. So, \eqref{eq: SubGrad} is also P-matrix for
 $$\tilde\alpha\le\tilde\alpha^*_{2n+1} := \min_{
 	\begin{matrix}
 		{\scriptstyle 1\le k\le 2n+1}\\
 		{\scriptstyle 1\le j\le 2n+1}
 	\end{matrix}
 }\tilde\alpha_{kj},$$ 
 for all $j=1,\ldots,2n+1$.\\
 Hence, by the Theorem 2 of \cite{gale1965jacobian},
 for the subvector $\tilde x_j = (x_1,x_{j+1},x_{n+j+1})'$ we have
\begin{equation*}
	\nabla_{_{\lambda,H_j,\sigma_j}}\tilde\rho_j(\tilde\alpha)\cdot\tilde x_j>0,
\end{equation*}
and so for $\theta\in\Upsilon_{2n+1}$
\begin{align*}
&\sigma_j^2\frac{\p\rho_{\lambda,H_j}}{\p\lambda}((k-1)\alpha)\,x_1
+ \sigma_j^2\frac{\p\rho_{\lambda,H_j}}{\p H_j}((k-1)\alpha)\,x_{j+1} \\
&+ 2\sigma_j\frac{\p\rho_{\lambda,H_j}}{\p \sigma_j}((k-1)\alpha)\,x_{n+j+1}>0.
\end{align*}
Consequently, for
\begin{equation}\label{eq: StepLimit}
	\alpha\le\frac{\tilde\alpha^*_{2n+1}}{2n+1},
\end{equation}
  \eqref{eq: VecProd} is positive for all $k=1,\ldots,2n+1$, and so the vector
$\nabla_\theta\brho_{\theta,2n+1}(\alpha)\,x$
is positive. Hence, by the Theorem 2 of \cite{gale1965jacobian}, the gradiant
$\nabla_\theta\brho_{\theta,2n+1}(\alpha)$
is a P-matrix, and so the mapping \eqref{eq: Mapping} is injective on $\Upsilon_{2n+1}$.
\end{proof}

Now, we can obtain the strong consistency for the GMM estimators for the mmfOU processes.

\begin{thm}\label{crl 3.2}
For the GMM parameter estimator \eqref{estimator} for the mmfOU process $U^n$, we have
$$
\hat{\theta}_N=(\hat{\lambda}_N;\hat{\bH}_N;\hat{\bsigma}_N)'\overset{a.s.}{\longrightarrow}\theta_0=(\lambda_0;\bH_0;\bsigma_0).
$$
\end{thm}
\begin{proof}
As an straightforward result of Theorem \ref{thm: cncty}, Assumption 2.1 of \cite{Barboza-Viens-2017} is valid and we can apply Theorem 2.1 in \cite{Barboza-Viens-2017}.
\end{proof}
\begin{rem}\label{rem: consist_constrains}
	First, we note that for $\lambda\in[0,\e^{\Psi(1)}]$ the consistency that we studied above is valid. Second, the proof of Theorem \ref{thm: cncty} returns an important result as the nature of numerical estimation. By \eqref{eq: StepLimit}, the larger the number of parameters we aim to estimate, the smaller the sufficient step $\alpha>0$ we need. However, in any case, there exists some sufficient $\alpha$ for which the consistency is garanteed.
\end{rem}
\section{Asymptotic Normality}\label{sect:normality}
For $\theta=(\lambda ;\bH ;\bsigma)'\in\R^{2n+1}$, the spectral density function (SDF) of $U^n$ is
\begin{equation}\label{eq: SDF_mmfOUn}
	f_{\theta}(x) = f(\theta;x) = \sum_{i=1}^n\sigma_i^2 f_{\lambda,H_i}(x),
\end{equation}
where
\begin{equation}\label{fou-sd}
f_{\lambda,H_i}(x)
=
\frac{\sin(\pi H_i)\Gamma(1+2H_i)}{2\pi}\, \frac{|x|^{1-2H_i}}{x^2+ \lambda^2},
\end{equation}
is the SDF of the fOU process $U^{\lambda,H_i}$ (see \cite{almani_sottinen_2023}).
Denote
\begin{equation*}
\bar{f}_\theta(x):=\sum_{p\in\Z}f_\theta\left(x+\frac{2p\pi}{\alpha}\right)\,;\quad x\in\left[-\frac{\pi}{\alpha},\frac{\pi}{\alpha}\right],
\end{equation*}
and for $\l_i,\l_j\in\lb 1,\ldots,2n+1\rb$
\begin{equation*}
R_\theta(x|i,j):=\min(1,|x|^{l_i+l_j})\bar{f}_\theta(x).
\end{equation*}
Then, as explained in \cite{Barboza-Viens-2017}, the GMM estimator has asymptotic normality if $R_{\theta_0}(\cdot|i,j)\in L^2(-\pi,\pi)$. For the process $U^n$ we have
\begin{eqnarray*}
R_\theta(x|i,j)
&=&\min(1,|x|^{l_i+l_j})\overline{f}_{\theta}(x)\notag\\
&=&\sum_{k=1}^n\sigma_k^2 R_{\lambda,H_k}(x|i,j),
\end{eqnarray*}
where
\begin{gather*}
	R_{\lambda,H_k}(x|i,j):=\min(1,|x|^{l_i+l_j})\overline{f}_{\lambda,H_k}(x),\\
	\bar{f}_{\lambda,H_k}(x):=\sum_{p\in\Z}f_{\lambda,H_k}\left(x+\frac{2p\pi}{\alpha}\right)\,;\quad x\in\left[-\frac{\pi}{\alpha},\frac{\pi}{\alpha}\right].
\end{gather*}
So, $R_\theta(\cdot|i,j)\in L^2(-\pi,\pi)$ if and only if $R_{\lambda,H_k}(\cdot|i,j)\in L^2(-\pi,\pi)$ for all $k=1,\ldots,n$. The following Lemma is a generalization of  Lemma 3.1 in \cite{Barboza-Viens-2017}.


\begin{lmm}\label{lmm 3.3}
For the $U^n$ process $R_{\theta_0}(\cdot|i,j)\in L^2(-\pi,\pi)$ for
\begin{itemize}
\item[ (i)] $l_i+l_j\geq 1$, if all $H_k\in(0,1)$,
\item[(ii)] $l_i+l_j=0$, if all $H_k\in(0,\frac{3}{4})$, i.e. $H_{\sup}\le 3/4$,
\end{itemize}
where $n\ge 1$ ($L\geq 3$).
\end{lmm}

Now, for $\theta\in\Theta$ we define
\begin{align*}
&\hat{G}_N(\theta):=\nabla_\theta\hat{\bg}_N(\theta),\\
&G(\theta):=\E[\nabla_\theta\bg(\cdot,\theta)],
\end{align*}
and note that $G(\theta)$ does not depend on time $t$ because of the stationarity of $U^n$. 

Finally, we have the following theorem for asymptotic normality of the GMM estimator of $U^n$.

\begin{thm}
The GMM estimator \eqref{estimator} for the mmfOU processes $U^n$ satisfies
\begin{itemize}
\item[(i)] $\E[\Vert\hat{\theta}_N-\theta_0\Vert^2]=O(N^{-1})$,
\item[(ii)] $N^c\Vert\hat{\theta}_N-\theta_0\Vert\overset{a.s.}{\longrightarrow}0$ for $c<\frac12$,
\item[(iii)] $\sqrt{N}(\hat{\theta}_N-\theta_0)\overset{d}{\longrightarrow}\mathcal{N}\left(0,C(\theta_0)\Lambda C(\theta_0)'\right)$,
\end{itemize}
where $C(\theta_0)=[G(\theta_0)'AG(\theta_0)]^{-1}G(\theta_0)'A$, and $\Lambda$ is a $L\times L$ matrix (for $L\ge 2n+1$) such that
\begin{gather*}
\Lambda_{ij}=2\sum_{p\in\Z}\left[\sum_{k=0}^L b_k^{i,j}\rho_{\theta_0}\left(\alpha\cdot(k+p)\right)\right]^2,\\
b_k^{i,j}=\sum_{q=0}^{L-k}a_q^i a_{q+k}^j + \sum_{q=0}^{L-k}a_q^j a_{q+k}^i.
\end{gather*}
\end{thm}

\begin{proof}
Theorem \ref{thm: cncty} and Lemma \ref{lmm 3.3} yield that in fact the Assumptions 1.2-1.3 of \cite{Barboza-Viens-2017} are valid. So, we can apply Lemma 2.2 and Theorem 2.3 of \cite{Barboza-Viens-2017}.
\end{proof}

\section{Simulation}\label{sect:simulation}

In this section, we perform a simulation study to test our GMM estimator for the parameters vector $\theta =(\lambda ;\bH ;\bsigma)$ for $n=3$. The statistical error metrics we use are
\begin{align*}
	\widehat{MSE}&:=\frac{1}{m}\sum_{i=1}^m\Vert\hat{\theta}_{N,i}-\theta\Vert^2,\\
	e(\widehat{Var})&:=\text{maximum eigenvalue of } \widehat{Var}(\theta_N),\\
	\widehat{Bias}^2&:=\left\Vert\frac{1}{m}\sum_{i=1}^m\hat{\theta}_{N,i}-\theta\right\Vert^2,
\end{align*}
where $\widehat{Var}(\theta_N)$ is the empirical covariance matrix of the estimated parameters based on the $m$ replications. For statistically high accuracy (cooking), we setup $m = 500$ replications in our estimations.\\

To have the consistency, as mentioned by Remark \ref{rem: consist_constrains}, we take $\lambda = \exp(\Psi(1))$. For \eqref{eq: StepLimit} to hold we take $N = 200, 600,$ and $1000$ in time interval $[0,1]$. Also, Theorem \ref{thm: cncty} is valid when $\bsigma$ has a fixed sign. So, we choose the positive-sign values $\bsigma = (0.5, 1, 1.5)'$.\\

For asymptotic normality, we are restricted by the Lemma \ref{lmm 3.3}. So, we take filters of positive orders $l = 1,2,\ldots,2n+1$. On the other hand, we are interested to test our estimator for varying Hurst values. Specially, for applications sense, we are interested to include the standard Brownian motion (Bm), the case $H_j = 1/2$ for some $j$. Hence, we take $\bH = (0.3, 0.5, 0.7)'$.\\

In each estimation, we first simulate replications for the $U^n$ process, and then estimate its parameters $\theta =(\lambda ;\bH ;\bsigma)$.
To produce $U^n$, we use iterations of the Langevin equation 
\begin{equation}\label{eq: Langevin_Un}
	\d U^n_t = -\lambda U^n_t\d t + \d\left(\sum_{i=1}^n \sigma_i B^{H_i}_t\right),
\end{equation}
at $N$ points of time interval $[0,1]$ with uniform step-sizes. In each replication, we use the autocovariance function $\rho_{\lambda,n}$ given in \eqref{eq: rho_practical} to estimate the parameters $\hat{\theta}_{N}=(\hat{\lambda}_{N};\hat{\bH}_{N};\hat{\bsigma}_{N})$. To do this, we employ finite difference filters of orders $L = 2n+1$ (number of parameters) with $A=I_{2n+1}$ the identity matrix to minimize the calculation cost (see \cite{Barboza-Viens-2017}).\\

 The resulting values of the statistical error metrics for our simulations are given in Table \ref{tab:Metrics}. As one can check, the method produces reasonably accurate estimations and the errors vanishe while inreasing the time step accuracy by increasing $N$. This fact is nicely observable in Figure \ref{fig:Estimations}, confirming our analytical formula \eqref{eq: StepLimit} in action.

\begin{table}[H]
	\centering
	\begin{tabular}{l|ccc}
		\toprule
		$N$     & 200   & 600   & 1000 \\
		\midrule[1.5pt]
		$\widehat{MSE}$   & \,0.006867\, & \,0.001149\, & \,0.000558\, \\
		\midrule
		$e(\widehat{Var})$ & 0.004257 & 0.000564 & 0.000242 \\
		\midrule
		$\widehat{Bias}^2$ & 0.00082 & 0.000294 & 0.000184 \\
		\bottomrule[1.5pt]
	\end{tabular}%
	\label{tab:Metrics}%
	\caption{Statistical errors for estimations by $m=500$ replications.}
\end{table}%
\begin{figure}[H]
	\includegraphics[scale=0.6]{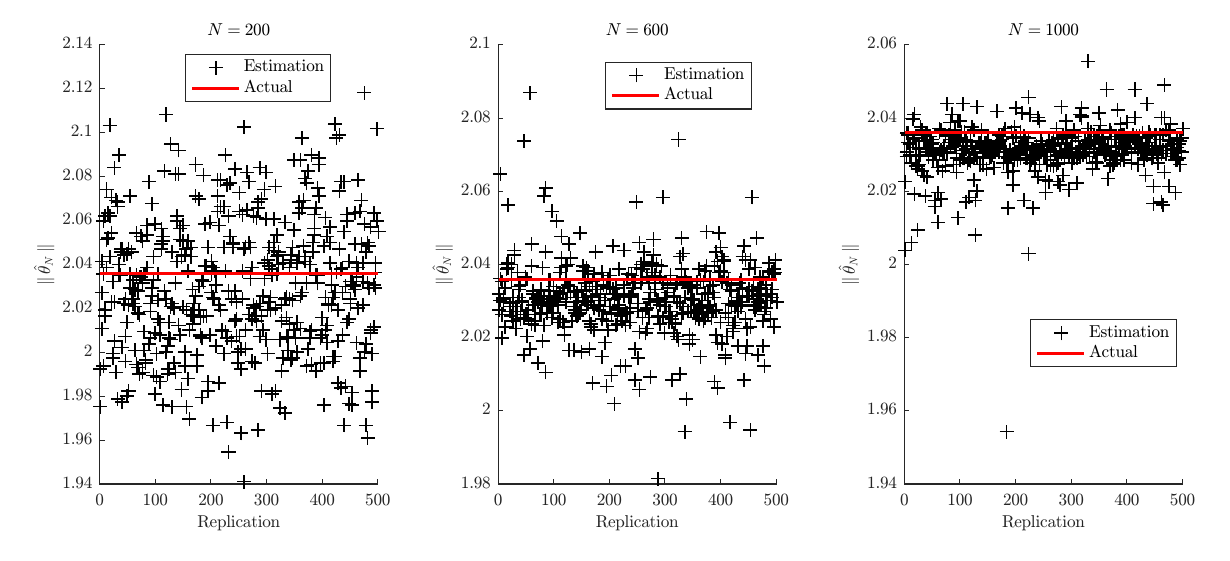}
	\label{fig:Estimations}
	\caption{Estimations of $\theta =(\lambda ;\bH ;\bsigma)$ for $m=500$ replications.}
\end{figure}
\bibliographystyle{siam}
\bibliography{pipliateekki.bib}

\end{document}